\documentclass[12pt]{amsart}
\usepackage[T2A]{fontenc}
\usepackage[cp1251]{inputenc}

\usepackage{mathrsfs,latexsym,cite,amsmath,amsfonts,amssymb,amsthm,euscript}
\usepackage[dvips]{graphicx}
\usepackage[english]{babel}
\theoremstyle{plain}
\newtheorem{Theorem}{Theorem}
\newtheorem{Lemma}{Lemma}
\begin{document}
\title[Povzner--Wienholtz-type theorems]{Povzner--Wienholtz-type theorems for Sturm--Liouville operators with singular coefficients}

\author[A. Goriunov, V. Mikhailets, V. Molyboga] {Andrii Goriunov, Vladimir Mikhailets, Volodymyr Molyboga}

\email[Andrii Goriunov]{goriunov@imath.kiev.ua}

\email[Vladimir Mikhailets]{mikhailets@imath.kiev.ua}

\email[Volodymyr Molyboga]{molyboga@imath.kiev.ua}

\address{Institute of Mathematics of NAS of Ukraine \\
         Tereshchenkivska str., 3 \\
         Kyiv \\
         Ukraine \\
         01024}
\thanks{The authors were supported by the European Union’s Horizon 2020 research and innovation programme 
	under the Marie Sklodowska-Curie grant agreement No 873071 
	(SOMPATY: Spectral Optimization: From Mathematics to Physics and Advanced Technology).} 

\keywords{self-adjointness problem; Sturm-Liouville operator; singular coefficients}

\begin{abstract}
We introduce and investigate symmetric operators $L_0$ associated in the complex Hilbert space $L^2(\mathbb{R})$ 
with a formal differential expression  
\[l[u] :=-(pu')'+qu + i((ru)'+ru')  \]
under minimal conditions on the regularity of the coefficients. 
They are assumed to satisfy conditions 
\[q=Q'+s;\quad \frac{1}{\sqrt{|p|}}, \frac{Q}{\sqrt{|p|}}, \frac{r}{\sqrt{|p|}} \in L^2_{loc}\left(\mathbb{R}\right), \quad s \in L^1_{loc}\left(\mathbb{R}\right), \quad\frac{1}{p}\neq 0\,\,\text{a.e.,} \]
where the derivative of the function $Q$ is understood in the sense of distributions, and all functions $p$, $Q$, $r$, $s$ are real-valued.
In particular, the coefficients $q$ and $r'$ may be Radon measures on $\mathbb{R}$, while function $p$ may be discontinuous. 
The main result of the paper are constructive sufficient conditions on the coefficient $p$ 
which provide that the operator $L_0$ being semi-bounded implies it being self-adjoint. 
\end{abstract}

\maketitle

\section{Introduction}
The problem of symmetric operators being self-adjoint is one of the main problems in the theory of differential operators 
and serves as the basis for the analysis of their spectral properties and scattering problems 
(see, e.~g.~\cite{Reed-book1}). 
Investigation of this problem for the Sturm--Liouville and Schr\"{o}dinger operators in the spaces $L^2(\mathbb{R}^n)$
is inspired by the problems of mathematical physics and has numerous applications. 
The results that were obtained for this problem in the case of regular coefficients are rather complete. 
Hartman \cite{Hart} and  Rellich \cite{Rellich} were the first to show that 
boundedness from below of the operator 
generated by the Sturm--Liouville differential expression
\[l(u) :=-(pu')'+qu \]
in the Hilbert space $L^2(\mathbb{R})$ together with the integral condition on the function $p>0$ 
(which is obviously satisfied if $p\equiv 1$)  
are sufficient for the minimal operator to be self-adjoint.

Povzner \cite{Povzner} and later independently Wienholtz \cite{Wienholtz} established that preminimal operator 
\[ \mathrm{L}_{00}u = -\Delta u + qu, \qquad u \in C_0^\infty (\mathbb{R}^n) \]
being semi-bounded implies it being essentially self-adjoint in the space $L^2(\mathbb{R}^n)$, 
if function $q$ is real-valued and continuous. 
Later conditions on the regularity of potential $q$ were significantly weakened in papers \cite{Hinz,Schmincke,Simader1,Simader2}. 

A more general result regarding the self-adjointness of second-order symmetric elliptic operators with smooth coefficients
was received by Yu.~M.~Berezanskii \cite[ch.~VI.1.7]{Berezanskii-book}. 
It states that the semi-bounded differential operator $\mathrm{L}_{00}$ defined on $C_0^\infty(\mathbb{R}^n)$ 
is essentially self-adjoint in the Hilbert space $L^2(\mathbb{R}^n)$, 
if the condition of a globally finite rate of propagation is satisfied,
that is, each solution to a hyperbolic differential equation
\[u_{tt} + Lu = 0 \]
which has compact support for $t=0$ has compact support for any $t>0$. 
These results were developed further in papers \cite{Berezanskii-Samoilenko,Orochko,Rofe-Beketov},  
see also \cite{Braverman} and the bibliography therein. 

At the same time, due to problems of the mathematical physics there arose and in the recent years has been increasing 
an interest towards differential operators with singular coefficients 
which are not locally summable functions  (see \cite{Albeverio,Teschl,Mikhailets0} and the references therein). 
Analysis of such operators presents significant mathematical difficulties, 
since their domain does not allow an explicit description and 
may not contain smooth functions other then zero. 
Moreover, the correct definition of such operators in the case of singular coefficients
is a non-trivial problem which may be of considerable interest for the theory of differential operators and its applications. 

In the present paper we solve this problem for the Sturm--Liouville operators, 
applying an approach based on the theory of quasi-differential operators  (see \cite{Zettl,Goriunov1,Goriunov2,Eckhardt}). 
This approach is natural from an analytical point of view, 
since it allows to use the techniques and methods of the theory of ordinary differential equations.
Furthermore, we show that the minimal operator $L_0$ generated by the formal differential expression $l$ 
has a dense in $L^2(\mathbb{R})$ domain and is symmetric.
Therefore, the question naturally arises about the conditions for it to be self-adjoint. 
Substantial answer to it is provided by two new theorems of the Povzner--Wingoltz type, which are formulated in Section 2.
The proofs of these theorems are given in Section 4 of this paper. 
It uses a number of auxiliary statements outlined in Section 3. 
 
\section{Povzner--Wienholtz-type results}
We consider operators generated by the differential Sturm--Liouville expression 
\begin{equation}\label{SL}
l[u] :=-(pu')'+qu + i((ru)'+ru')
\end{equation}
with real coefficients $p$, $q$ and $r$ given on $\mathbb{R}$. 
If these coefficients are regular enough, then the mapping 
\[ \mathrm{L}_{00}: u \mapsto l[u], \qquad u \in C_0^\infty(\mathbb{R}) \]
defines a densely defined in the complex Hilbert space $L^2(\mathbb{R})$ 
preminimal symmetric operator $\mathrm{L}_{00}$. 
Here naturally arises question whether the closure of this operator 
$\mathrm{L}_0:=\left(\mathrm{L}_{00}\right)^{\sim}$ is self-adjoint. 
A large number of papers are devoted to this problem (see, e. g. references in \cite{Zettl-book}). 
For instance, Hartman \cite{Hart} and  Rellich \cite{Rellich} established that 
if operator $\mathrm{L}_{00}$ is bounded from below and  
\[r\equiv 0, \qquad 0<p\in C^2(\mathbb{R}), \qquad q \text{ is piecewise continuous on } \mathbb{R},\]
and function $p$ satisfies condition 
\begin{equation}\label{cond_HR}
\int_{0}^{\infty}p^{-1/2}(t)d\,t=\int_{-\infty}^{0}p^{-1/2}(t)d\,t=\infty, 
\end{equation}
then minimal operator $\mathrm{L}_0$ corresponding to $l$ is self-adjoint. 
In the paper \cite{Stetkaer-Hansen}  the conditions on the regularity of the coefficients of $l$ were weakened: 
\[r\equiv 0, \qquad 0<p \text{ is locally Lipschitz,} \qquad q\in L^2_{loc}(\mathbb{R}).\]
Another sufficient condition for the operator  $\mathrm{L}_0$ to be self-adjoint was obtained in \cite{Clark}. 
It may be written in the form 
\begin{equation}\label{cond_CG}
\|p\|_{L^\infty(-\rho, -\rho/2)}, \|p\|_{L^\infty(\rho/2, \rho)} = O(\rho^2), \qquad \rho \rightarrow \infty.
\end{equation}
Here the coefficients of \eqref{SL} satisfy conditions 
\[r\equiv 0, \qquad  0<p \in W^1_{2,loc}(\mathbb{R}), \qquad q \in L^1_{loc}(\mathbb{R}).\]
Examples show that conditions \eqref{cond_HR} and \eqref{cond_CG} are independent (see \cite{Clark}). 

We will assume throughout what follows that the assumptions 
\begin{equation}\label{cond_GMM}
q=Q'+s;\quad \frac{1}{\sqrt{|p|}}, \frac{Q}{\sqrt{|p|}}, \frac{r}{\sqrt{|p|}} \in L^2_{loc}\left(\mathbb{R}\right), \quad s \in L^1_{loc}\left(\mathbb{R}\right),   \quad\frac{1}{p}\neq 0\,\,\text{ a. e.,}
\end{equation}
hold, where the derivative of function $Q$ is understood in the sense of distributions 
and all the coefficients $p$, $Q$, $s$, $r$ are real-valued functions.

We propose to consider the operators generated by the formal differential expression \eqref{SL} as quasi-differential operators, 
which are defined applying compositions of differential operators with locally summable coefficients.
These operators are defined using the Shin--Zettl matrix function 
specifically chosen to correspond to the coefficients of $l$ 
(see \cite{Zettl,Goriunov1,Goriunov2,Eckhardt}).

In our case it has the form
\begin{equation}\label{SZ_matrix}
A(x)= \begin{pmatrix}
\frac{Q+ir}{p}&\frac{1}{p}\\
-\frac{Q^2+r^2}{p} + s&-\frac{Q-ir}{p}
\end{pmatrix} 
\end{equation}
and due to our assumptions belongs to the class $L^1_{loc}(\mathbb{R}, \mathbb{C}^2)$.

It can be used to define corresponding quasi-derivatives as follows: 
\begin{align}\label{quasideriv}
u^{[0]}:=u,\qquad  u^{[1]}:=pu' - (Q+ir)u, \qquad
u^{[2]}:=\left(u^{[1]}\right)'+\frac{Q-ir}{p}u^{[1]}+\left(\frac{Q^2+r^2}{p} - s\right)u.
\end{align}

Formal differential expression \eqref{SL} may now be defined as quasi-differential:
\begin{equation*}
l[u]:=-u^{[2]},\quad \mathrm{Dom}(l):=\left\{u:\mathbb{R}\rightarrow \mathbb{C}\left|\,u,\,u^{[1]}\in
\mathrm{AC}_{loc}(\mathbb{R})\right.\right\}.
\end{equation*}
This definition is motivated by the fact that 
\begin{equation*}\label{eq_20}
\langle -u^{[2]},\varphi\rangle=\langle-(pu')'+qu + i((ru)'+ru'),\varphi\rangle\qquad \forall\varphi\in
\mathrm{C}_{0}^{\infty}(\mathbb{R})
\end{equation*}
in the sense of distributions.

We define for the quasi-differential expression $l$ the operators $\mathrm{L}$ and $\mathrm{L}_{00}$ as:
\begin{align*}
\mathrm{L}u:=l[u],\qquad 
&\mathrm{Dom}(\mathrm{L}):= 
\left\{u\in L^2(\mathbb{R})\,\left|\,u,\,u^{[1]}\in \mathrm{AC}_{loc}(\mathbb{R}),l[u]\in   L^{2}(\mathbb{R})\right.\right\},\\
\mathrm{L}_{00}u:=\mathrm{L}u, \qquad 
&\mathrm{Dom}(\mathrm{L}_{00}):=\left\{u\in \mathrm{Dom}(\mathrm{L})\,  \left|\,\mathrm{supp}\,u\Subset\mathbb{R}\right.\right\}. \hspace{60pt}
\end{align*}
The operators $\mathrm{L}$ and $\mathrm{L}_{00}$ are maximal and preminimal operators for expression $l$ respectively. 
Their definitions coincide with the classical ones if the coefficients $l$ are sufficiently smooth. 
Below in Section \ref{sec_PrlRs} we will show that the operator $\mathrm{L}_{00}$ is densely defined in $L^2(\mathbb{R})$ and is symmetric. 

Let us formulate the main results of the paper in the form of two theorems. 
The first of them is a natural generalization of the above-mentioned result of Hartman and Rellich. 

\begin{Theorem}\label{th_B}
Let the coefficients of the formal differential expression \eqref{SL} satisfy the assumptions \eqref{cond_GMM} 
and also 
\begin{align*}
\mathtt{(i)}\hspace{5pt} & p\in W^1_{2,loc}(\mathbb{R}), \hspace{5pt}  p>0, \\
\mathtt{(ii)}\hspace{5pt} & \int_{-\infty}^{0}p^{-1/2}(t)d\,t=\int_{0}^{\infty}p^{-1/2}(t)d\,t=\infty. \hspace{250pt}
\end{align*}
Then, if operator $\mathrm{L}_{00}$ is bounded from below, then it is essentially self-adjoint 
and $\mathrm{L}_{00}^* = \mathrm{L} = \mathrm{L}^*$. 
\end{Theorem}

For the case $p\equiv 1$, $r\equiv 0$ Theorem~\ref{th_B} was previously established in~\cite{Mikhailets1}. 

In the second theorem, additional conditions on the coefficient $p$ are imposed not on the entire axis,
but only on a sequence of finite intervals. 
However, outside of these intervals the function $p$ may vanish and be discontinuous. 

\begin{Theorem}\label{th_C}
Suppose the assumptions \eqref{cond_GMM} are satisfied and the operator $\mathrm{L}_{00}$ is bounded from below. 
Suppose the sequence of intervals $\Delta_n:=[a_n,b_n]$ exists such that 
\[\quad -\infty< a_{n}<b_{n}<\infty, \quad b_{n}\rightarrow -\infty,\;n\rightarrow -\infty,\quad a_{n}\rightarrow \infty,\;n\rightarrow\infty,\]
where the coefficients $p$ satisfy the additional conditions 
\begin{itemize}
	\item [$\mathtt{(i)}$]  $p_{n}:=p|_{\Delta_{n}}\in W^1_2(\Delta_n)$, $p_n>0$; 
	\item [$\mathtt{(ii)}$] $\exists C>0:\; p_n(x)\leq C|\Delta_n|^2$, $n\in\mathbb{Z}$,
	where $|\Delta_n|$ is the length of interval~$\Delta_n$. 
\end{itemize}
	
Then operator $\mathrm{L}_{00}$ is essentially self-adjoint and $\mathrm{L}_{00}^* = \mathrm{L} = \mathrm{L}^*$. 
\end{Theorem}

For the case $p\equiv 1$, $r\equiv 0$ necessary and sufficient conditions for semi-boundedness of operator $\mathrm{L}_{00}$ 
were obtained in \cite{Mikhailets3}. 

\section{Preliminary results}\label{sec_PrlRs}
Let us formulate some auxiliary statements that are used in the proof of Theorems \ref{th_B} and~\ref{th_C}. 
Some of them may be of interest by themselves. 
Let $\Delta=(\alpha,\beta)$ be a finite interval in $\mathbb{R}$.

Consider the Sturm--Liouville operators generated by the formal differential expression \eqref{SL} 
in the complex Hilbert space $L^{2}(\Delta)$. 
Our assumptions \eqref{cond_GMM} imply that 
\[ q=Q'+s,\quad \frac{1}{\sqrt{|p|}}, \;\frac{Q}{\sqrt{|p|}},\; \frac{r}{\sqrt{|p|}} \in L^{2}(\Delta),\quad 
s \in L^{1}(\Delta), \quad \frac{1}{p}\neq 0\quad\text{a. e. on } \Delta.\]
Therefore, the Shin--Zettl matrix of the form \eqref{SZ_matrix} is correctly defined, 
which defines corresponding quasi-derivatives by the formulas \eqref{quasideriv}.

Let us now consider a quasi-differential expression $l$ on the interval $\Delta$ 
\begin{align*}
l_{\Delta}[u] & :=-u^{[2]},\quad
\mathrm{Dom}(l_\Delta):=\left\{u\in L^2(\Delta)\left|\,u,\,u^{[1]}\in
\mathrm{AC}(\overline{\Delta})\right.\right\}, 
\end{align*}
and the corresponding minimal and maximal operators
\begin{align*}
\mathrm{L}_\Delta u & :=l_\Delta[u], \qquad
\mathrm{Dom}(\mathrm{L}_\Delta) :=
\left\{u\in L^2(\Delta)\left|u,\,u^{[1]}\in  \mathrm{AC}(\overline{\Delta}),l_\Delta[u]\in L^2(\Delta)\right.\right\}, \\
\mathrm{L}_{0\Delta}u & :=\mathrm{L}_\Delta u, \qquad
\mathrm{Dom}(\mathrm{L}_{0\Delta}) :=\left\{u\in \mathrm{Dom}(\mathrm{L}_\Delta) \,\left|\,u^{[j]}(\alpha)=u^{[j]}(\beta)=0,\; j=0,1\right.\right\},
\end{align*}

\begin{Lemma}[Theorems~9, 10 \cite{Zettl}]\label{lm_Zettl}
Operator $\mathrm{L}_{0\Delta}$ has a domain which is dense in the space $L^{2}(\Delta)$ and 
\begin{align*}
\mathtt{(i)}\hspace{5pt} & (\mathrm{L}_{0\Delta})^* = \mathrm{L}_\Delta, \\
\mathtt{(ii)}\hspace{5pt} & \mathrm{L}_{0\Delta} = (\mathrm{L}_\Delta)^*.
\end{align*}
\end{Lemma}

Let us now pass from the finite interval $\Delta$ to the entire axis. 

\begin{Theorem}\label{th_A}
The following statements hold for operators $\mathrm{L}$ and $\mathrm{L}_{00}$.
\begin{itemize}
\item [$1^{0}$.] Operator $\mathrm{L}_{00}$ has a domain which is dense in the Hilbert space $L^{2}(\mathbb{R})$.
\item [$2^{0}$.] The equality holds: 
\[ (\mathrm{L}_{00})^*=\mathrm{L}. \]
\item [$3^{0}$.] Operator $\mathrm{L}$ is closed and preminimal operator $\mathrm{L}_{00}$ is closable.
\item [$4^{0}$.] The domain of the minimal operator $\mathrm{L}_{0}$ may be represented as 
 \[ \mathrm{Dom}(\mathrm{L}_{0}) =\left\{u\in \mathrm{Dom}(\mathrm{L}) \left|\,[u,v]_{-\infty}^{\infty}=0\quad \forall
	v\in  \mathrm{Dom}(\mathrm{L})\right.\right\},\]
where 
\begin{align*}
[u,v](t)\equiv [u,v] & :=u(t)\overline{v^{[1]}(t)}-u^{[1]}(t)\overline{v(t)}, \\
[u,v]_{a}^{b} & :=[u,v](b)-[u,v](a),\quad -\infty\leq a\leq b\leq \infty.
\end{align*}
\item [$5^{0}$.] Operator $\mathrm{L}_{0}$ is symmetric.
\item [$6^{0}$.] Let $p,1/p\in L^\infty(\Delta)$. Then
\begin{align*}
\mathtt{i)}\hspace{5pt} Q, r\in L^2(\Delta), \qquad
\mathtt{ii)}\hspace{5pt} u\in \mathrm{Dom}(\mathrm{L}) \Rightarrow u|_\Delta\in W_2^1(\Delta).
\end{align*}
In particular, if  $p,1/p\in L_{loc}^\infty(\mathbb{R})$, then 
\begin{align*}
\mathtt{i)}\hspace{5pt} Q, r\in L_{loc}^2(\mathbb{R}), \qquad
\mathtt{ii)}\hspace{5pt} \mathrm{Dom}(\mathrm{L})\subset W_{2,loc}^1(\mathbb{R}).
\end{align*}
\item [$7^{0}$.] Let $p,1/p\in L_{loc}^\infty(\mathbb{R})$. Then
\begin{equation*}
(\mathrm{L}_{00}u,u)_{L^2(\mathbb{R})}=\int_{\mathbb{R}}p|u'|^{2}d\,x-\int_{\mathbb{R}}Qd\,|u|^{2}+\int_{\mathbb{R}}s|u|^
{2}d\,x,\qquad u\in \mathrm{Dom}(\mathrm{L}_{00}).
\end{equation*}
\end{itemize}
\end{Theorem}

Before proving Theorem \ref{th_A}, we prove several statements we will need. 

\begin{Lemma}\label{lm_Lagrange}
For arbitrary functions $u,v\in \mathrm{Dom}(\mathrm{L})$ on a finite interval $[a,b]$ the relations hold:
\begin{equation}\label{Lagrange}
\int_{a}^{b}l[u]\overline{v}d\,x-\int_{a}^{b}u\overline{l[v]}d\,x=[u,v]_{a}^{b}.
\end{equation}
\end{Lemma}
\begin{proof}
The statement is proved by a direct computation applying integration by parts:
\begin{align*}
\int_{a}^{b}l[u]\overline{v}d\,x-&\int_{a}^{b}u\overline{l[v]}d\,x = \\
& =\int_{a}^{b}\left(-(u^{[1]})'\overline{v}-\frac{Q-ir}{p}u^{[1]}\overline{v}-\frac{Q^2+r^2}{p}u\overline{v}+s u\overline{v}\right)d\,x+ \\
& \qquad +\int_{a}^{b}\left(u(\overline{v^{[1]}})'+\frac{Q+ir}{p}u\overline{v^{[1]}}+\frac{Q^2+r^2}{p}u\overline{v}-s
u\overline{v}\right)d\,x = \\
&  =\left.\left(u\overline{v^{[1]}}-u^{[1]}\overline{v}\right)\vphantom{\int}\right|_{a}^{b} +\int_{a}^{b}\left(u^{[1]}\overline{v}'-\frac{Q-ir}{p}u^{[1]}\overline{v}-u'\overline{v^{[1]}}
+\frac{Q+ir}{p}u\overline{v}^{[1]}\right)d\,x=\\
&  =[u,v]_{a}^{b} +\int_{a}^{b}\left(\frac{1}{p}u^{[1]}\overline{v^{[1]}}-\frac{1}{p}u^{[1]}\overline{v^{[1]}}\right)d\,x
=[u,v]_{a}^{b}.
\end{align*}
Lemma is proved.
\end{proof}

\begin{Lemma}\label{lm_uv_form_lim}
For arbitrary functions $u,v\in \mathrm{Dom}(\mathrm{L})$ the following limits exist and are finite:
\begin{equation*}\label{uv_form_lim}
[u,v](-\infty):=\lim_{t\rightarrow-\infty}[u,v](t),\qquad [u,v](\infty):=\lim_{t\rightarrow\infty}[u,v](t).
\end{equation*}
\end{Lemma}
\begin{proof}
Let us fix the number $b$ in the equality \eqref{Lagrange} and pass in it to the limit as $a\rightarrow-\infty$. 
Since, by the assumptions of the Lemma, the functions  $u,v,l[u],l[v]\in L^{2}(\mathbb{R})$, 
then limit $[u,v](-\infty)$ exists and is finite.
The existence of the limit $[u,v](\infty)$ is proved similarly.
\end{proof}

\begin{proof}[Proof of Theorem~\ref{th_A}]
The proofs of assertions $1^{0}$ -- $4 ^{0}$ are similar to the proofs for the case of the semiaxis \cite{Zettl,Naimark}. 
They rely on similar properties of operators on a finite interval.
For the convenience of the reader, we present these proofs in full. 

$1^{0}$. Let the function $h\in L^2(\mathbb{R})$ be orthogonal to $\mathrm{Dom}(\mathrm{L}_{00})$. 
Let us show that $h\equiv0$.

Let $\Delta$ be an arbitrary fixed finite interval. 

Let us define by the expression $l$ the quasi-differential operators $\mathrm{L}_{0\Delta}$ and $\mathrm{L}_\Delta$. 
The space $L^{2}(\Delta)$ may be embedded into $L^{2}(\mathbb{R})$, 
assuming that outside the interval $\overline{\Delta}$ function $u\in L^{2}(\Delta)$ is equal to zero.  
Thus, domain $\mathrm{Dom}(\mathrm{L}_{0\Delta})$ of the operator $\mathrm{L}_{0\Delta}$ 
becomes a part of $\mathrm{Dom}(\mathrm{L})$.
Moreover, the function $u\in \mathrm{Dom}(\mathrm{L}_{0\Delta})$ extended in such a way belongs to $\mathrm{Dom}(\mathrm{L}_{00})$.

Thus, the function $h$ will be orthogonal to $\mathrm{Dom}(\mathrm{L}_{0\Delta})$. 
Due to Lemma~\ref{lm_Zettl} the domain $\mathrm{Dom}(\mathrm{L}_{0\Delta})$ is dense in the space $L^2(\Delta)$. 
Therefore $h|_{\overline{\Delta}} = 0$ almost everywhere. 

As the interval $\Delta\subset \mathbb{R}$ is arbitrary, this implies that $h=0$ almost everywhere on $\mathbb{R}$.

Assertion $1^{0}$ is proved. 

$2^{0}$. Due to assertion $1^{0}$ for the operator $\mathrm{L}_{00}$ the adjoint operator $(\mathrm{L}_{00})^*$ exists.

Let $u\in \mathrm{Dom}(\mathrm{L}_{00})$, $v\in \mathrm{Dom}(\mathrm{L})$. 
Then applying the Lagrange identity from Lemma~\ref{lm_Lagrange} we get:
\[ (\mathrm{L}_{00}u,v)_{L^2(\mathbb{R})}=(u,\mathrm{L}v)_{L^2(\mathbb{R})}.\]
Hence $\mathrm{L}\subset (\mathrm{L}_{00})^*$. 
Let us prove the reverse inclusion. 

Let $v$ be an arbitrary element from $\mathrm{Dom}\left((\mathrm{L}_{00})^*\right)$  
and let $\Delta=(\alpha,\beta)$ be a fixed finite interval in $\mathbb{R}$. 
Then 
\[\left((\mathrm{L}_{00})^*v,u\right)_{L^2(\mathbb{R})}=\left(v,\mathrm{L}_{00}u\right)_{L^2(\mathbb{R})}\qquad \forall u\in \mathrm{Dom}(\mathrm{L}_{0\Delta}). \]
Since $u|_{\mathbb{R}\setminus\overline{\Delta}} = 0$, 
then inner products may be expressed as integrals over the interval $\overline{\Delta}$, 
i. e. they are inner products in $L^{2}(\Delta)$ and 
\begin{equation}\label{eq_Prp24}
\left(((\mathrm{L}_{00})^*v)_\Delta,u\right)_{L^2(\Delta)}=\left(v_\Delta,\mathrm{L}_{0\Delta}u\right)_{L^2(\Delta)}\qquad
\forall u\in \mathrm{Dom}(\mathrm{L}_{0\Delta}),
\end{equation}
where $((\mathrm{L}_{00})^*v)_\Delta$, $v_\Delta$ are the restrictions of corresponding functions onto interval~$\overline{\Delta}$.
Equality \eqref{eq_Prp24} due to Lemma~\ref{lm_Zettl} implies that 
\[   ((\mathrm{L}_{00})^*v)_\Delta=\mathrm{L}_\Delta v_\Delta=\left(l[v]\right)_\Delta. \]
Since the interval $\Delta$ is arbitrary it follows that 
\begin{equation*}
v\in \mathrm{Dom}(\mathrm{L}), \qquad (\mathrm{L}_{00})^*v=l[v]=\mathrm{L}v,
\end{equation*}
which we needed to prove.

Thus assertion $2^{0}$ is proved.

$3^{0}$. Immediately follows from assertion $2^{0}$.

$4^{0}$. Since the operator $\mathrm{L}_0$ is closed and $ (\mathrm{L}_0)^*=\mathrm{L}$, 
then obviously $\mathrm{Dom}\left(\mathrm{L}_0\right)$ consists of those and only those functions 
$u\in \mathrm{Dom}\left(\mathrm{L}\right)$ that satisfy relation
\begin{equation*}
(u,\mathrm{L}v)_{L^{2}(\mathbb{R})}=(\mathrm{L}u,v)_{L^{2}(\mathbb{R})}\qquad \forall v\in \mathrm{Dom}\left(\mathrm{L}\right),
\end{equation*}
which due to Lagrange identity is equivalent to assertion $4^{0}$.

Assertion $4^{0}$ is proved.

$5^{0}$. Lemma~\ref{lm_Lagrange} implies that operator $\mathrm{L}_{00}$ is symmetric. 
Therefore operator $\mathrm{L}_0$ being its closure is also symmetric.

$6^{0}$. Let the conditions of the assertion be satisfied. 
They immediately imply that $\sqrt{|p|} \in L_{loc}^\infty(\mathbb{R})$. 
Due to assumptions \eqref{cond_GMM} functions $\frac{Q}{\sqrt{|p|}}, \frac{r}{\sqrt{|p|}}\in L_{loc}^2(\mathbb{R})$.
Therefore, if $p \in L^\infty(\Delta)$ then
\begin{equation*}
\sqrt{|p|}\cdot \frac{Q}{\sqrt{|p|}}=Q\in L^2(\Delta),\quad \sqrt{|p|}\cdot \frac{r}{\sqrt{|p|}}=r\in L^2(\Delta).
\end{equation*}

Let, further, $u\in \mathrm{Dom}(\mathrm{L})$. 
Definition of the domain $\mathrm{Dom}(\mathrm{L})$ implies that 
\[  p\left(u'- \frac{Q+ir}{p}u\right) = pu' - (Q+ir)u \in AC_{loc}(\mathbb{R})\subset L_{loc}^2(\mathbb{R}). \]
Therefore, taking into account that $\frac{1}{p}\in L^\infty(\Delta)$ we get:
\begin{equation*}
\frac{1}{p}\cdot pu'=u'\in L^2(\Delta),\quad\text{ i. e. }\quad u_\Delta\in W_2^1(\Delta).
\end{equation*}

Assertion $6^{0}$ is proved.

$7^{0}$. Considering that, due to our assumptions 
$Q, r\in L_{loc}^2(\mathbb{R})$ and $\mathrm{Dom}(\mathrm{L}_{00})\subset W_{2, comp}^1(\mathbb{R})$, 
for $u\in\mathrm{Dom}(\mathrm{L}_{00})$ we have:
\begin{align*}
(\mathrm{L}&_{00}u,u)_{L^2(\mathbb{R})} 
=\int_{\mathbb{R}}\left(-(u^{[1]})'\overline{u}-\frac{Q-ir}{p}u^{[1]}\overline{u}-\frac{Q^2+r^2}{p}|u|^2+s|u|^2\right)d\,x=\\
& =\int_{\mathbb{R}}\left((u^{[1]}\overline{u}'-\frac{Q-ir}{p}u^{[1]}\overline{u}-\frac{Q^2+r^2}{p}|u|^2+s|u|^2\right)d\,x=\\
&=\int_{\mathbb{R}}\left((pu'-(Q+ir)u)\overline{u}'-\frac{Q-ir}{p}(pu'-(Q+ir)u)\overline{u}-\frac{Q^2+r^2}{p}|u|^2+s|u|^2\right)d\,x =\\
& =\int_{\mathbb{R}}\left(p|u'|^2-Q(u'\overline{u}+u\overline{u}')+s|u|^2\right)d\,x
=\int_{\mathbb{R}}p|u'|^2 d\,x-\int_{\mathbb{R}}Q d|u|^2+\int_{\mathbb{R}}s|u|^2 d\,x.
\end{align*}

Assertion $7^{0}$ and therefore Theorem~\ref{th_A} are proved.
\end{proof}

\section{Proofs of Theorems \ref{th_B} and \ref{th_C}}\label{sec:Proofs}
To prove Theorems \ref{th_B} and \ref{th_C} we need two more auxiliary statements 
regarding local regularity of functions from domains of minimal and maximal operators. 

\begin{Lemma}\label{lm_phi_u}
Let the initial assumptions about the coefficients~\eqref{cond_GMM} be satisfied 
and suppose $p>0$ and $p\in W_{2,loc}^1(\mathbb{R})$. 
Then for an arbitrary function $\varphi\in W_{2,comp}^{2}(\mathbb{R})$ and any function $u\in \mathrm{Dom}(\mathrm{L})$ 
we have $\varphi u\in \mathrm{Dom}(\mathrm{L}_{00})$. 
\end{Lemma}

\begin{proof}
Conditions $p\in W_{2, loc}^1(\mathbb{R})$ and $p>0$ imply that continuous on $\mathbb{R}$ function $p$ 
is locally separated from zero. 
Therefore function $\frac{1}{p}\in C(\mathbb{R})$ and is locally bounded on $\mathbb{R}$. 
Therefore, according to assertion~$6^{0}$ of Theorem~\ref{th_A} 
\begin{equation}\label{eq_Prp28}
Q, r\in L_{loc}^2(\mathbb{R}), \qquad \mathrm{Dom}(\mathrm{L}) \subset W_{2,loc}^1(\mathbb{R}).
\end{equation}
	
Due to \eqref{eq_Prp28} we have:
\begin{align*}
\mathtt{(i)}\hspace{5pt} & \varphi u\in \mathrm{AC}_{comp}(\mathbb{R}),\hspace{345pt} \\
\mathtt{(ii)}\hspace{5pt} & (\varphi u)^{[1]}=\left(p(\varphi u)'-(Q+ir)\varphi u\right)=p\varphi' u+ \varphi u^{[1]}\in \mathrm{AC}_{comp}(\mathbb{R}), \\
\mathtt{(iii)}\hspace{5pt} & l[\varphi u]=\varphi l[u]-\varphi' u^{[1]} - p'\varphi' u - p\varphi'' u - p\varphi' u' - (Q-ir) \varphi' u \in L_{comp}^{2}(\mathbb{R}), \\
\mathtt{(iv)}\hspace{5pt} & \mathrm{supp}\,\varphi u \text{ is a compact set}.
\end{align*}
Thus $\varphi u\in \mathrm{Dom}(\mathrm{L}_{00})$.

Therefore Lemma is proved.
\end{proof}

\begin{Lemma}\label{lm_phi_u_int}
Let the assumptions~\eqref{cond_GMM} about the coefficients of $l$ be satisfied, 
suppose numbers $a'<a<b<b'$ and positive functions $p|_{[a',a]} \in W_2^1([a',a])$, $p|_{[b,b']} \in W_2^1([b,b'])$, 
and suppose function $\varphi\in C^{2}(\mathbb{R})$ is such that 
\begin{align*}
\mathtt{(i)}\hspace{5pt} & 0 \leq\varphi(x) \leq 1,\quad x\in \mathbb{R},\hspace{295pt} \\
\mathtt{(ii)}\hspace{5pt} & \varphi(x) = \begin{cases}
1, & x\in [a,b];\\
0, &x\in (\infty, a')\bigcup (b',\infty).
\end{cases}
\end{align*}
Then for every function $u\in \mathrm{Dom}(\mathrm{L})$ we have $\varphi u\in \mathrm{Dom}(\mathrm{L}_{00})$. 
\end{Lemma}

\begin{proof}
Assumptions of Lemma imply that function $p$ is continuous and separated from zero on the set $[a',a]\bigcup[b,b']$. 
Therefore function $\frac{1}{p}$ is bounded on the compact set $[a',a]\bigcup[b,b']$ 
due to assertion~$6^{0}$ of Theorem~\ref{th_A}. 
Considering that $\varphi'(x) = 0$, $x\in (\infty, a')\bigcup [a,b]\bigcup (b',\infty)$ we have:
\begin{align*}
\mathtt{(i)}\hspace{5pt} & \varphi u\in \mathrm{AC}_{comp}(\mathbb{R}),\hspace{345pt} \\
\mathtt{(ii)}\hspace{5pt} & (\varphi u)^{[1]}=\left(p(\varphi u)'-(Q+ir)\varphi u\right)=p\varphi' u+ \varphi u^{[1]}\in \mathrm{AC}_{comp}(\mathbb{R}), \\
\mathtt{(iii)}\hspace{5pt} & l[\varphi u]=\varphi l[u]-\varphi' u^{[1]} - p'\varphi' u - p\varphi'' u - p\varphi' u' - (Q-ir) \varphi' u \in L_{comp}^{2}(\mathbb{R}), \\
\mathtt{(iv)}\hspace{5pt} & \mathrm{supp}\,\varphi u\text{ is a compact set}.
\end{align*}
Thus $\varphi u\in \mathrm{Dom}(\mathrm{L}_{00})$ and Lemma is proved.
\end{proof}

\begin{proof}[Proof of Theorem~\ref{th_B}]
Without loss of generality, we assume that 
\begin{equation}\label{eq_AB20}
(\mathrm{L}_0u,u)_{L^2(\mathbb{R})}\geq (u,u)_{L^2(\mathbb{R})},\qquad u\in \mathrm{Dom}(\mathrm{L}_0).
\end{equation}

To prove that operator $\mathrm{L}_0$  is self-adjoint, it is sufficient (and necessary) to show that every function 
that satisfies conditions 
\begin{equation}\label{eq_AB22}
\mathrm{L}v=0, \qquad v\in \mathrm{Dom}(\mathrm{L})
\end{equation}
is equal to zero almost everywhere. 

Let equality  \eqref{eq_AB22} hold. 
We choose a function sequence $\{\varphi_n\}_{n\in \mathbf{N}}$ so that following conditions are satisfied: 
\begin{itemize}
	\item [$\mathtt{(i)}$] $0\leq\varphi_n(x)\leq 1$, 
	\item [$\mathtt{(ii)}$] $\varphi_n(x)=1$, $x\in [-n,n]$;
	\item [$\mathtt{(iii)}$] $\mathrm{supp}\,\varphi_n\subset [-n-1,n+1]$;
	\item [$\mathtt{(iv)}$] $|\varphi_n'|\leq K$, where the constant $K$ does not depend on $n$. 
\end{itemize}
Consider now the function 
\begin{equation*}
\rho(x):=\begin{cases}
-&\int_x^0 p^{-1/2}(t)d\,t,\quad x\in (-\infty,0],\hspace{25pt} \\
&\int_0^x p^{-1/2}(t)d\,t,\quad x\in [0,\infty). 
\end{cases}
\end{equation*}
	
Then, taking into account condition $ii)$ of Theorem \ref{th_B} the composition $\varphi_n(\rho)$ 
belongs to the space $W_{2,comp}^2(\mathbb{R})$, $n\in\mathbb{N}$. 
Therefore, functions $\varphi_n(\rho)v$ satisfy inequality \eqref{eq_AB20}. 
After substitution we obtain: 
\begin{equation*}
K^2\int_{n\leq|\rho(x)|\leq n+1}|v(x)|^2d\,x\geq \int_{|\rho(x)|\leq n}|v(x)|^2d\,x.
\end{equation*}
Taking into account that $v\in L^2(\mathbb{R})$ and passing to the limit in the last inequality, we obtain that $v= 0$.

Theorem is proved.
\end{proof}

\begin{proof}[Proof of Theorem~\ref{th_C}]
Without loss of generality, we assume that the intervals $\Delta_n=[a_n,b_n]$ do not intersect and that $b_{-1}<0$ and $a_1>0$. 
Let us show that every function $v$ that satisfies the equality $Lv=0$ is equal to $0$ almost everywhere on $\mathbb{R}$. 

Let $\{\varphi_n\}_{n\in \mathbb{N}}$ be a sequence of real infinitely differentiable functions 
with following properties:
\begin{align*}
\mathtt{(i)} \quad & 0\leq\varphi_n(x)\leq 1, x\in\mathbb{R}, n\in\mathbb{N}; \hspace{270pt} \\
\mathtt{(ii)} \quad & \varphi_n(x)=1,\, x\in [b_{-n},a_n]; \\
\mathtt{(iii)} \quad & \mathrm{supp}\,\varphi_n \subset (a_{-n},b_n); \\
\mathtt{(iv)} \quad & \exists K>0: \quad  |\varphi_{n}'(x)|\leq
\begin{cases}
K|\Delta_{-n}|^2,\quad & \text{if}\quad x<0, \\
K|\Delta_{n}|^2,\quad & \text{if}\quad x>0.
\end{cases}
\end{align*}

Due to Lemma \ref{lm_phi_u_int} function $\varphi_{n}v\in \mathrm{Dom}(\mathrm{L}_{00}) \subset \mathrm{Dom}(\mathrm{L}_0)$ and  
$\mathrm{supp}\,\varphi_{n}'\subset \Delta_{-n}\cup \Delta_{n},\quad n\in \mathbb{N}$.
\begin{equation}\label{eq_AB57}
(\mathrm{L}_{00}\varphi_{n}v,\varphi_{n}v)_{L^{2}(\mathbb{R})}
=\int_{\mathbb{R}}p(\varphi_{n}')^{2}|v|^{2}d\,x
+\int_{\mathbb{R}}p\varphi_{n}\varphi_{n}'(v\overline{v}'-v'\overline{v})d\,x.
\end{equation}
Due to our assumptions 
\begin{equation*}
\mathrm{Re}\,\int_{\mathbb{R}}p\varphi_{n}\varphi_{n}'(v\overline{v}'-v'\overline{v})d\,x=0,
\end{equation*}
Therefore, from inequality \eqref{eq_AB20} we have that 
\begin{equation}\label{eq_AB60}
\int_{b_{-n}}^{a_{n}}|v|^{2}d\,x\leq C K^{2} \int_{\Delta_{-n}\cup \Delta_{n}}|v|^{2}d\,x, 
\end{equation}
where the constant $C>0$ is the one from assumption $\mathtt{(ii)}$ of Theorem~\ref{th_C}. 
	
Taking into account that $v\in L^{2}(\mathbb{R})$ and passing in inequality \eqref{eq_AB60} to the limit as  $n\rightarrow\infty$ 
we conclude that $v= 0$.
	
Theorem is proved.
\end{proof}

A similar approach can also be applied to study symmetric Sturm--Liouville operators defined on the semiaxis.

\end{document}